\documentclass[letterpaper,12pt]{amsart}
\usepackage{
color,graphicx,amscd,amsmath,amssymb,amsxtra}
\pdfoutput=1
\usepackage[pdftex]{hyperref}
\usepackage{abbrViro}
minus 1.03\fontdimen4
{Defining 
relations for reflections. I}
\author{Oleg Viro}
\dedicatory{Stony Brook University, NY, USA;\break
PDMI, St.\ Petersburg, Russia}
\address{Department of Mathematics, Stony Brook University,
Stony Brook NY, 11794-3651, USA
}
\email{oleg.viro@gmail.com}

\begin{document}

\begin{abstract} 
An idea to present a classical Lie group of positive dimension 
by generators and relations sounds dubious, but happens to be fruitful. 
The isometry groups of classical geometries admit elegant and useful 
presentations by generators and relations. They are closely related 
to geometry and allow to make fast and efficient calculations in the 
groups. In this paper simple presentations of  
the isometry groups of Euclidean plane, 2-sphere, the real projective 
plane and groups $SO(3)$ and $O(n)$ are introduced. 
\end{abstract}

\maketitle

\section{Isometries of the Euclidean plane}\label{s1}

\subsection{Well-known facts}\label{s1.1} See, e.g., \cite{Berger}.

{\bfit Classification. } Any 
isometry of the Euclidean plane is either the identity $\id$, 
or a reflection in a line, or a translation, 
or a rotation about a point, 
or a glide reflection (composition of a reflection in a line and a
translation along the same line).

{\bfit Generators. }
Any isometry is a composition of at most three reflections in lines.

{\bfit Compositions of refections. }
The composition $R_m\circ R_l$ of reflections $R_l$ and $R_m$ in 
lines $l$ and $m$, respectively, is \begin{itemize}
\item the identity if $l=m$;
\item a translation if $l\parallel m$;
\item a rotation if $l$ is transverse to $m$.
\end{itemize}

If $l\parallel m$, then the translation $R_m\circ R_l$ moves any point
by a distance twice greater than the distance between the lines in
the direction perpendicular to these lines.  

If $l$ is transverse to $m$, then the composition $R_m\circ R_l$ is a
rotation about the intersection point $m\cap l$ by the angle twice greater
than the angle between the lines. See figure \ref{f1}.

\begin{figure}[htb]
\centerline{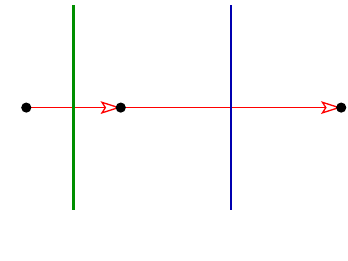 \hspace{3cm} 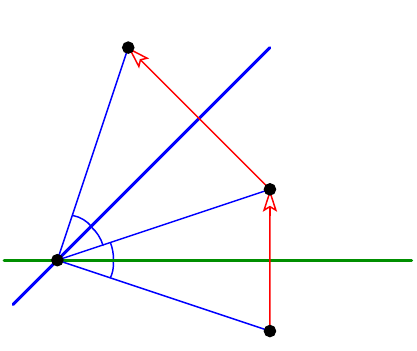}
\caption{Compositions of two reflections}
\label{f1}
\end{figure}

Obviously, in this way any translation and any rotation can be presented 
as a composition of two reflections. 

\subsection{Relations among reflections}\label{s1.2}
Presentations of a translation or a rotation as a composition of two
reflections are not unique. 

Any pair of
lines perpendicular to the direction of a translation and having the same
distance from each other gives a presentation of the same translation.

Any pair of lines meeting at the center of the
rotation and forming the same angle gives a presentation of the same
rotation.

This non-uniqueness of presentations can be formulated as relations among
reflections in lines. Each such relation involves four reflections and has
form 
\begin{equation}\label{eq1}R_m\circ R_l=R_{m'}\circ R_{l'}.\end{equation}
The lines $l$, $m$, $l'$ and $m'$ belong to the same pencil: either all
four lines are parallel, or all four pass through the same point.

The pencils of lines have natural metrics: in a pencil of parallel lines
this is the usual distance between parallel lines; in a pencil of lines
having a common point this is the usual angle between lines. In either case, for 
the lines  $l$, $m$, $l'$ and $m'$ involved in \eqref{eq1}  
 the corresponding distances are related: 
\begin{equation}\label{eq1'}
\dist(l,m)=\dist(l',m') \ \text{ and } \ \dist(l,l')=\dist(m,m').
\end{equation}
\begin{Th}\label{distRel}
Relations \eqref{eq1'} follow from \eqref{eq1}. 
\end{Th}
\begin{proof}
Indeed, 
$$\dist(l,R_m\circ R_l(l))=2\dist(l,m) \text{ and } 
\dist(l',R_{m'}\circ R_{l'}(l'))=2\dist(l',m'), $$ 
therefore if $R_m\circ R_l=R_{m'}\circ R_{l'}$, then $R_m\circ R_l$ should
move any line of the pencil by the same distance as $R_{m'}\circ R_{l'}$.
Thus, $\dist(l,m)=\dist(l',m')$.

Further, by multiplying \eqref{eq1} by $R_{m'}$ from the left and by $R_{l}$
from the right, we obtain $R_{m'}\circ R_{m}=R_{l'}\circ R_{l}$, from which the
second relation of \eqref{eq1'} is deduced exactly as the first is deduced from
\eqref{eq1}.
\end{proof}

Any three lines $l$, $m$, $l'$ that belong to a pencil, can be supplemented by a
unique line $m'$ belonging to the same pencil such that the relations \eqref{eq1'} 
are satisfied. (Of course, here the type of distance is determined by the type 
of the pencil.) Belonging of the lines to the same pencil and relations
\eqref{eq1'} imply \eqref{eq1}. 

Certainly, relations \eqref{eq1} are well-known. For example, in the
group-theoretic approach to foundations of the classical geometries (see,
e.g., the monograph \cite{Bachmann} by Bachmann) lines were identified with 
reflections in them, belonging 
of three lines to a pencil of lines was defined as the fact that the
composition of the corresponding reflections is a reflection. So, the
relations \eqref{eq1} were turned into a definition of concurrency of three
lines, that is belonging three lines to a pencil. 
We will call \eqref{eq1}  {\sfit pencil relations\/}.

If $m=l$, then  \eqref{eq1'} implies $m'=l'$ and then \eqref{eq1} 
turns into $R_l^2=R_{l'}^2$. The latter 
follows from the fact that reflections are involutions:
\begin{equation}\label{eq2}
R_l^2=\id.
\end{equation}
We will call \eqref{eq2} {\sfit involution relations.\/}

\subsection{Completeness of relations}\label{s1.3}
Although the relations \eqref{eq1} and \eqref{eq2} are well-known,
to the best of my knowledge, the following theorem has not
appeared in the literature.

\begin{Th}\label{mainThE(2)}
Any relation among reflections in the group of isometries of Euclidean
plane is a corollary of the pencil and involution relations. 
\end{Th}

\begin{lem}\label{lem4refl}
Any composition of four reflections can be converted by pencil and
involution relations into a composition of two reflections.  
\end{lem}

\begin{proof}
Consider a composition  $R_n\circ R_m\circ R_l\circ R_k$ of four reflections.
If any two consecutive lines coincide (i.e., $k=l$, or $l=m$, or $m=n$), 
then, by applying the involution relation, we can eliminate the
corresponding reflection from the composition. So, in the rest of the
proof, we assume that none of consecutive lines coincide.

Assume that $k\cap l\ne\empt$ and $m\cap n\ne\empt$.  
By rotating the pairs of lines $k,l$ and $m,n$ about the intersection
points, obtain pairs of lines $k',l'$ and $m',n'$ such that 
$l'=m'$ is the line connecting the points $k\cap l$ and $m\cap n$.\\
\vspace{5pt}
\centerline{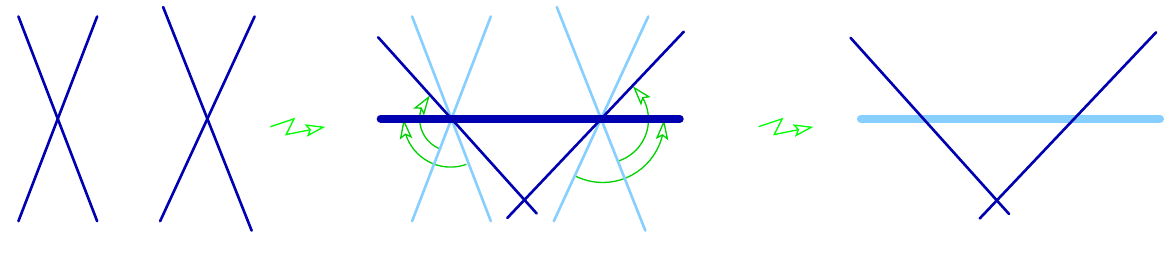}
By the pencil relation, $R_n\circ R_m\circ R_l\circ R_k=R_{n'}\circ R_{m'}
\circ R_{l'}\circ R_{k'}$, and, by the involution relation, $R_{m'}\circ
R_{l'}=\id$. Hence,   $R_n\circ R_m\circ R_l\circ R_k=R_{n'}\circ R_{k'}$.

Assume that $k\parallel l$ and $m\cap n\ne\empt$. Then translate $k\cup l$
to $k'\cup l'$ such that $l'$ passes through $m\cap n$, and rotate $m\cup n$
about the point $m\cap n$ to $m'\cup n'$ such that $m'=l'$. \\
\vspace{5pt}
\centerline{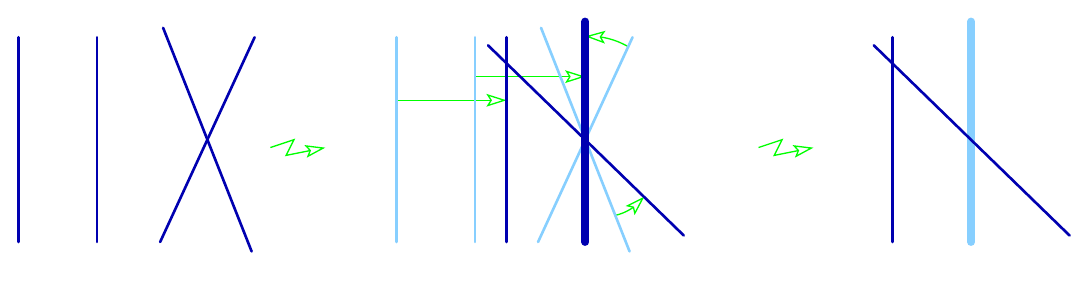}
By the pencil
relations,  $R_n\circ R_m\circ R_l\circ R_k=R_{n'}\circ R_{m'}
\circ R_{l'}\circ R_{k'}$, and, by the involution relation, $R_{m'}\circ
R_{l'}=\id$. Hence,   $R_n\circ R_m\circ R_l\circ R_k=R_{n'}\circ R_{k'}$.

If $m\parallel n$ and $k\cap l\ne\empt$, then we do the same, but exchanging
the r\^oles of pairs $k,l$ and $m,n$.

Assume that  $k\parallel l$ and  $m\parallel n$, but $l\cap m\ne\empt$.
Then by rotating the middle pair of lines $l,m$ by right angle we obtain
the situation that was already considered:   $k\cap l'\ne\empt$ and 
$m\cap n'\ne\empt$.\\ 
\vspace{5pt}
\centerline{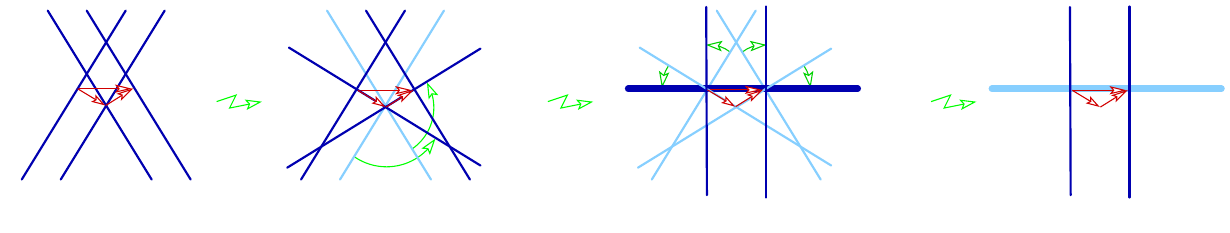}
The figure above provides also a proof of a well-known
fact that composition of translation is a translation by the vector which
is the sum of the vectors corresponding to the original translations.\\ 

If all the lines are parallel, then by a translation of $k\cup l$ such that
the image of $l$ would coincide with $m$ and applying pencil and involution
relations as above, we can reduce the number of reflections.
\end{proof}

\begin{proof}[\bfseries Proof of Theorem \ref{mainThE(2)}]
Take any relation $R_1\circ R_2\circ \dots\circ R_n=\id$ among reflections.
By Lemma \ref{lem4refl}, we may reduce by applying relations \eqref{eq1}
and \eqref{eq2} its length $n$ to a number which is less than four. It cannot
be three, because a composition of an odd number of reflections reverses the
orientation, and hence cannot be equal to the identity. 
The only composition of two reflections which is the identity is of the
form \eqref{eq2}.
\end{proof}

Lemma \ref{lem4refl} provides an effective geometric simplification 
algorithm for evaluating a composition of reflections. 

\section{Isometries of the 2-sphere}\label{s1+}

\subsection{Well-known facts}\label{s1+.1}
The group of isometries  of the 2-sphere $S^2=\{x\in\R^3: |x|=1\}$
coincides
with the orthogonal group $O(3)$: any isometry $S^2\to S^2$ is a
restriction of linear orthogonal transformation of $\R^3$. In this
section we prefer the language of isometries of $S^2$, but everything can
be easily translated to the language of $O(3)$.

{\bfit Classification. } Any isometry of $S^2$ is either the identity
$\id$, or a reflection $R_l$ in a great circle $l$ (i.e., the restriction of a
reflection $\R^3\to\R^3$ in a 2-subspace), or a rotation about a pair of
antipodal points (i.e., the restriction of a rotation $\R^3\to\R^3$ about a
1-subspace), or a glide reflection (the restriction of the composition of a
rotation about a 1-subspace and the reflection in the orthogonal
2-subspace). 

{\bfit Generators. } Any isometry of $S^2$ is a composition of at most
three reflections in great circles.

{\bfit Compositions of reflections. } The composition $R_m\circ R_l$ of
reflections in great circles $l$ and $m$ is 
\begin{itemize} 
\item the identity if $l=m$;  
\item a rotation about $l\cap m$ by the angle twice greater than the angle
between $l$ and $m$ if $l\ne m$.
\end{itemize}

\subsection{Relations among reflections}\label{s1+.2}
Any rotation can be presented as a composition of two reflections. The
presentation is not unique. Any two great circles intersecting at the
antipodal points fixed by the rotation and forming the required angle give
rise to such a presentation. 

This non-uniqueness of presentation can be formulated as relations among
reflections in great circles. Each such relation involves four reflections
and looks like this:
\begin{equation}\label{eq1+}
R_m\circ R_l=R_{m'}\circ R_{l'}
\end{equation}
where $l,m,l',m'$ are great circles such that $l\cap m=l'\cap m'$ and the
angles between between the great circles satisfy two equalities:
$\measuredangle(l,m)=\measuredangle(l',m')$, and
$\measuredangle(l,l')=\measuredangle(m,m')$. 
Relations of this form are called the {\sfit pencil
relations,\/} like  similar relations in the isometry group of the Euclidean 
plane.

A reflection in a great circle is an involution, and we will
refer to the relations   $R_l^2=\id$ as to {\sfit
involution relations.\/}

\begin{Th}\label{mainThO(3)} 
In the isometry group of the 2-sphere, any relation among reflections 
follows from the pencil and involution relations.
\end{Th}

\begin{lem}\label{lem4reflO(3)}
In the isometry group of the 2-sphere, any composition of four reflections 
can be converted by pencil and involution relations into a composition of
two reflections.
\end{lem}

The proofs of \ref{lem4reflO(3)} and \ref{mainThO(3)} repeat (with obvious 
simplifications) the proofs of \ref{lem4refl} and \ref{mainThE(2)} given
above.

\section{Special orthogonal group $SO(3)$}\label{s3}

Special orthogonal group $SO(3)$ consists of linear maps $\R^3\to\R^3$ 
preserving distances and orientation.  Each such map has eigenvalue 1, i.e., 
it has fixed line and rotates the whole 3-space about the line by some
angle. Therefore it can be represented as a composition of two reflections
in planes. (Notice that reflections in planes do not belong to $SO(3)$,
because they reverse orientation.) 
The axis of rotation is the intersection of those planes, the
angle between the planes is half the rotation angle. 

\subsection{Reflections in lines}\label{s3.1}
The rotation by $180^\circ$ of the 3-space about a line $l$ is called the 
{\sfit reflection in $l$\/}. The notation $R_l$ is extended to reflections
of this kind. 

\begin{Th}\label{lemCompReflInLines}
A composition $R_b\circ R_a$ of reflections in lines $a$ and $b$ 
is the rotation by the angle twice greater than the angle
between $a$ and $b$ about the axis $c$ which is perpendicular to $a$ and $b$. 
\end{Th}

\begin{proof}
Denote by $C$ the plane containing $a$ and $b$. Denote by $A$ the plane
containing $a$ and $c$ and by $B$ the plane containing $b$ and $c$. 
Since $C\perp A$ and $C\cap A=a$, we have $R_a=R_A\circ R_C$. Similarly,
$R_b=R_B\circ R_C$. Further,
$$
R_b\circ R_a=R_B\circ R_C\circ R_A\circ R_C.
$$
Since the reflections in orthogonal planes $C$ and $A$ commute, 
$$
R_B\circ R_C\circ R_A\circ R_C=R_B\circ R_A\circ R_C^2=R_B\circ R_A.
$$
 As the composition of reflections in planes $A$, $B$ with $A\cap B=c$, 
$R_B\circ R_A$ is a rotation about $c$ by the angle twice the angle
between $A$ and $B$. The angle between $A$ and $B$ equals the angle between
$a$ and $b$.
\end{proof}

\subsection{Pencil relations}\label{s3.2}
The representation of a reflection in a line provided by Theorem 
\ref{lemCompReflInLines} is non-unique. This non-uniqueness can be
considered as relations among quadruples of reflections in lines. 
Namely, if $a$,
$b$, $c$, $d$ are coplanar lines, 
$\measuredangle(a,b)=\measuredangle(c,d)$ and
$\measuredangle(a,c)=\measuredangle(b,d)$, then 
\begin{equation}\label{eq3}
R_b\circ R_a=R_d\circ R_c.
\end{equation}

As \eqref{eq3} is similar to \eqref{eq1} and the four lines involved in 
\eqref{eq3} belong a pencil of lines, \eqref{eq3} is also called a 
{\sfit pencil relation\/}.

\subsection{Polar frame relations}\label{s3.3}
\begin{Th}\label{corOrthoFr}
If $a$, $b$ and $c$ are pairwise orthogonal lines in $\R^3$ passing 
through the origin, then 
$R_b\circ R_a=R_c$, or, equivalently and more symmetrically, 
\begin{equation}\label{eq4}
R_c\circ R_b\circ R_a=\id.
\end{equation}
\end{Th}

\begin{proof} By Theorem  \ref{lemCompReflInLines}, $R_b\circ R_a$
is the rotation about $c$ by $180^\circ$, that is $R_b\circ R_a=R_c$. Composing
both sides of this equality with $R_c$ and taking into account that
$R_c^2=\id$, we obtain \eqref{eq4}.
\end{proof}

A relation \eqref{eq4} is called an {\sfit polar frame relation.}  

\subsection{Presentation of $SO(3)$}\label{s3.4}
\begin{Th}\label{thPresSO3} The group 
$SO(3)$ is generated by reflections in lines, any relation in $SO(3)$ 
among reflections in lines follows from pencil, involution and 
polar frame relations.
\end{Th}

\begin{lem}\label{lem3refl}
Any composition of three reflections in lines of $\R^3$ can be converted by
polar frame, pencil and involution relations into a composition of 
two reflections in lines. 
\end{lem}

\begin{proof} Consider a composition $R_m\circ R_l\circ R_k$ of three
reflections. Let $P$ be a plane containing $l$ and $m$. Denote by $b$
the line orthogonal to $k$ and contained in $P$. (If $b$ is not orthogonal
to $P$, then 
$b=P\cap k^\perp$, where $k^\perp$ is the orthogonal complement to $k$; if
$P=k^\perp$, then $b$ can be any line in $P$.) Let $c$ be the line
orthogonal to $k$ and $b$. 

By a polar frame relation, $R_k=R_c\circ R_b$, and hence 
$R_m\circ R_l\circ R_k= R_m\circ R_l\circ R_c\circ R_b$. Now, lines $c$, $l$
and $m$ are coplanar, and by a pencil relation their composition 
$ R_m\circ R_l\circ R_c$ is a reflection in a line.
\end{proof}

\begin{proof}[\bfit Proof of Theorem \ref{thPresSO3}]
Any non-identity element of $SO(3)$ is a rotation
about a line, and, by Theorem \ref{lemCompReflInLines}, any rotation about a
line is a composition of two reflections in lines. Therefore $SO(3)$ is
generated by reflections in lines.

Take any relation $R_1\circ R_2\circ \dots\circ R_n=\id$ among reflections
in lines.
By Lemma \ref{lem3refl}, we may reduce its length $n$ by applying 
polar frame, pencil and involution relations to a number which is less than
three.  
The composition of two reflections in non-coinciding lines is a rotation by
the angle twice the angle between the lines. Thus, if the lines do not
coincide, then the composition is not the identity. If the lines coincide,
then the composition is the identity, but the relation is reduced to an 
involution relation.
A relation cannot consist of a single reflection in a line, because a
reflection in a line is not identity.
\end{proof}

\subsection{Orientation preserving isometries of the sphere}\label{s3.5}
Elements of $SO(3)$ are orientation preserving isometries of
$\R^3$. Their restrictions to the 2-sphere $S^2$ are orientation preserving 
isometries of $S^2$. Therefore the results of this section admit
reformulations for the group of orientation preserving isometires of the
2-sphere. In particular, Theorem  \ref{thPresSO3} means that this group is
generated by reflections in pairs of antipodal points and any relation
among such reflections follow from the corresponding versions of pencil,
involution and polar frame relations. 

This is in a sharp contrast to the situation in the group of orientation
preserving isometries of the Euclidean plane, in which the only involutions
are reflections in points, and they do not generate the group. The group
generated by all the involutions consists of the involutions themselves and
all the translations. 

\subsection{Isometries of the projective plane}\label{s3.6}
Consider the projective plane $\R P^2$ equipped with the metric defined by 
the Euclidean metric in $R^3$. 
The isometry group of the projective plane is $PO(3,\R)$. 
Each isometry $\R P^2\to\R P^2$ admits two liftings to the covering space 
$S^2$. The two liftings differ by the antipodal involution of $S^2$, the 
only non-trivial automorphism of the covering $S^2\to \R P^2$. The
antipodal involution of $S^2$ reverses orientation. Therefore one of the 
liftings of any isometry of $\R P^2$ is orientation reversing and the other
one is orientation reversing. 

The fixed point set of an isometry-involution of the projective plane
consists of a point and a line polar to each other. Its lifting reversing
orientation is a reflection in a great circle, while the lifting preserving
orientation is a reflection in a pair of antipodal points. The great circle
and the antipodal points cover the components of the original involution of
the projective plane. 

The construction of orientation preserving covering map defines an
isomorphism of the group of isometries of the projective plane to a
group of orientation preserving isometries of the 2-sphere. 

Therefore Theorem  \ref{thPresSO3} implies that the isometry group of the
projective plane is generated by involutions and any relation among the
involutions of the projective plane is a corollary of relations that
correspond to pencil, involution and polar frame relations.

\section{New graphical calculus for rotations}\label{s4}

\subsection{Arrow-arcs presenting a rotation}\label{s4.1}
Any orientation preserving isometry of $S^2$ is a rotation about a line 
$l\in\R^3$, or, if we want to speak solely in terms of $S^2$, 
a rotation about the set $l\cap S^2$ of two antipodal points. 
By Theorem \ref{lemCompReflInLines}, it is
a composition of reflections in lines $a$ and $b$ orthogonal to $l$. On
$S^2$, these reflections are represented by two-point sets 
$a\cap S^2$ and $b\cap S^2$. There is a natural ordering of the sets: if
the rotation is represented as 
$R_b\circ R_a$, we have to apply $R_a$, first, and $R_b$, second. Thus,
$a$ precedes $b$. 

A two-point set consisting of antipodal points can be recovered from 
any of its points. Let us pick up a point $A$ from $a$ and a point
$B$ from $b$ and form an ordered pair $(A,B)$. The pair $(A,B)$ encodes 
$R_l$. In order to make $(A,B)$ more visible,
let us connect them with an arc equipped with an arrow. The angular length
of the arc is half of the angle of the rotation.
\begin{figure}
\centerline{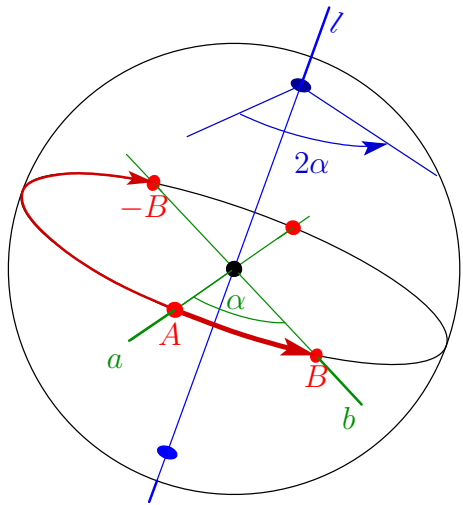}
\caption{Arrow-arc ${AB}$ 
represents a rotation by $2\alpha$.}
\label{f5}
\end{figure}

By pencil relations, an arrow-arc representing a rotation is defined by the
rotation up to gliding along its great circle and replacing the arrowhead
by the antipodal point.

This representation of rotation reminds the traditional representation of
a translation by an arrow. However, there are two important
differences. 

First, an arrow representing a translation connects a point with its
image under the translation, while in an arrow-arc representing a rotation 
the arrowhead is on the half a way to the image of the 
arrowtail.

Second, the image under any translation of an arrow representing a plane 
translation still represents the same translation, while an arrow-arcs
representing the same rotation are locked on the same great circle. 

Since a translation can be represented as a composition of two symmetries
with respect to points, a translation can be also represented by an arrow
connecting the centers of symmetries. This representation of translations
differs from the traditional one just by the length of the arrows: the
traditional arrows are twice longer. Despite this rescaling, it has all 
advantages of the traditional one.

\subsection{Triangle rule for arrow-arcs}\label{s3.6}
Usually a rotation of $\R^3$ is presented by so called  
{\sfit angular displacement\/} vector. 
This is an arrow directed along the axis of rotation,  its length is 
the angle of rotation and the direction is defined by some orientation
agreement (a right-hand rule). A well-known drawback of angular 
displacement vector 
is that it has a complicated behavior under composition of rotations: 
the angular dispalcement of a composition of rotations $U$ and $V$ 
is not the sum of the angular displacements of $U$ and $V$, although it is
defined by the angular displacements of $U$ and $V$. The relation between
the angular displacement of $V\circ U$  and the angular displacements of
$U$ and $V$ is too complicated to be useful.

The arrow-arc representation of $V\circ U$ can be easily calculated 
in terms of arrow-arcs for $U$ and $V$.

\begin{Th}\label{triangleRule}
Let rotations $U$ and $V$ of $S^2$ be represented by arrow-arcs $AB$ and
$CD$, respectively. The great circles containing the arcs intersect (as any
two great circles). By sliding the arrow-arcs along their great circles,
one can arrange them so that $B=C$. Assume this has been done.
Then the rotation $V\circ U$ is represented by
the arrow-arc $AD$.
\end{Th}

\begin{proof} Denote by $a$, $b$, $c$ and $d$ the lines connecting the
origin (i.e., the center of $S^2$) with the points $A$, $B$, $C$ and $D$.
Then $U=R_b\circ R_a$ and $V=R_d\circ R_c$, and $V\circ U=R_d\circ
R_c\circ R_b\circ R_a.$ Since $B=C$ and $b=c$, $R_c=R_b$, and 
 $V\circ U=R_d\circ R_c^2\circ R_a=R_d\circ R_a$.
\end{proof}

\begin{figure}[htb]
\centerline{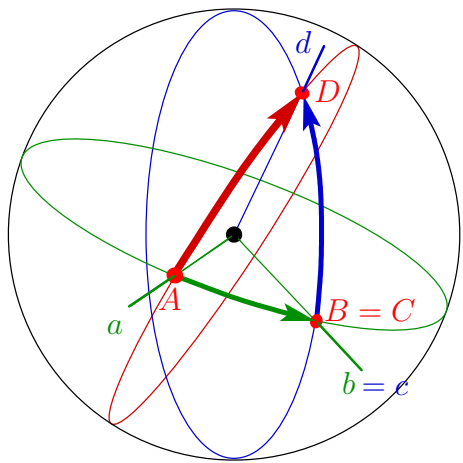}
\caption{The sum of arrow-arcs 
represents the composition of rotations represented by the summands.}
\label{f6}
\end{figure}

\section{Orthogonal groups $O(n)$}\label{s6}

\subsection{Well-known facts}\label{s6.1}
Any orthogonal linear transformation $T:\R^n\to\R^n$ splits into orthogonal
direct sum of orthogonal transformations of 1- and/or 2-dimensional spaces.
See, e.g., \cite{Berger}, 8.2.15. 
On a 1-dimensional subspace an orthogonal transformation is either identity
or symmetry in the origin. Therefore $T$ splits into
orthogonal direct sum of a reflection in a vector subspace and rotations of
2-subspaces. 

This orthogonal sum splitting of $T$ can be turned into a 
splitting of $T$ into 
a composition. For this,  extend the transformation on each summand 
to a transformation of the whole space by the identity on the orthogonal 
complement to the summand. 
The extended transformations of the whole space commute with each
other and their composition equals the original transformation.

Since a rotation of plane is a composition of two reflections in lines, it
follows that any element of $O(n)$ can be represented as a composition of
reflections in hyperplanes. Any element of $O(n)$ can be represented as a
composition of at most $n$ reflections. See, e.g., \cite{Berger}, 8.2.12.

\subsection{Pencil relations among reflections}\label{s6.2}
Non-uniqueness for representation as a composition of two reflections 
for a plane rotation implies non-uniqueness for representation of a
rotation of $\R^n$ about a subspace of codimension 2 as a composition of 
two reflections in hyperplanes. 

As in Section \ref{s1.2}, this non-uniqueness can be formulated as 
relations among reflections in hyperplanes. We will call these relations
also {\sfit pencil relations.\/}

\begin{Th}\label{mainThO(n)}
In $O(n)$ any relation among reflections in hyperplanes 
follows from pencil and involution relations. 
\end{Th}

A proof of Theorem \ref{mainThO(n)} is similar to the proof of Theorem
\ref{mainThE(2)} above and is based on the following lemma:

\begin{lem}\label{lemn+1refl}
In $O(n)$, any composition of $n+1$ reflections in hyperplanes 
can be converted by pencil and involution relations
 into a composition of $n-1$ reflections in hyperplanes. \qed  
\end{lem}



\begin{thebibliography}{99}

\bibitem{Bachmann} Friedrich Bachmann, Aufbau der Geometrie aus dem 
Spiegelungsbegriff, Springer-Verlag, 1973.

\bibitem{Berger} Marcel Berger, Geometry I, Springer, 2009.

\end{thebibliography}
\end{document}